%% file: crelle_maggi_revised.tex
\documentclass[11pt,a4paper,reqno]{amsart}

\usepackage{amsmath,amsfonts,amsthm,epsfig,graphicx}
\usepackage{graphicx,color}

\def\H{\mathcal{H}}

\def\N{\mathbb N}

\def\R{\mathbb R}

\def\C{\mathbf{C}}

\def\PHI{\mathbf{\Phi}}

\def\La{\Lambda}

\def\Om{\Omega}
\def\S{\Sigma}

\def\cl{{\rm cl}}                            
\def\X{\boldsymbol{\mathcal E}}

\def\a{\alpha}

\def\de{\delta}
\def\e{\varepsilon}

\def\om{\omega}
\def\vphi{\varphi}

\def\Lip{{\rm Lip}}
\def\II{{\rm II}}

\def\tr{{\rm tr}}
\def\Div{{\rm div}\,}
\def\Id{{\rm Id}\,}
\def\dist{{\rm dist}}
\def\diam{{\rm diam}}
\def\loc{{\rm loc}}
\def\diam{{\rm diam}}
\def\cof{{\rm cof}\,}
\def\spt{{\rm spt}}

\def\pa{\partial}

\def\cc{\subset\subset}

\def\C{\mathbf{C}}
\def\D{\mathbf{D}}

\def\exc{\mathbf{exc}}

\theoremstyle{plain}
\newtheorem{theorem}{Theorem}[section]
\newtheorem{lemma}[theorem]{Lemma}

\newtheorem{proposition}[theorem]{Proposition}
\newtheorem*{theorem*}{Theorem}
\newtheorem*{corollary*}{Corollary}

\theoremstyle{definition}
\newtheorem{definition}[theorem]{Definition}

\newtheorem{remark}[theorem]{Remark}

\newtheorem*{notation*}{Notation}

\numberwithin{equation}{section}
\numberwithin{figure}{section}

\title[Dimension of singular sets in  variational problems with free boundaries]{Dimensional estimates for singular sets in geometric variational problems with free boundaries}

\author[G. De Philippis]{Guido De Philippis}
\address{Institut f\"ur Mathematik, Universit\"at Z\"urich --
CH-8057 Z\"urich}
\email{guido.dephilippis@math.uzh.ch}

\author[F. Maggi]{Francesco Maggi}
\address{Department of Mathematics, The University of Texas at Austin,  2515 Speedway Stop C1200, Austin, Texas 78712-1202, USA}
\email{maggi@math.utexas.edu}

\begin{document}

\begin{abstract} We show that singular sets of free boundaries arising in codimension one anisotropic geometric variational problems  are $\mathcal H ^{n-3}$-negligible, where $n$ is the ambient space dimension. In particular our results apply to capillarity type problems, and establish everywhere regularity in the three-dimensional case.
\end{abstract}

\maketitle

\section{Introduction}  In \cite{dephilippismaggiCAPILLARI}, having in mind applications to capillarity problems and to relative isoperimetric problems, we studied the regularity of free boundaries in anisotropic geometric variational problems. The main result contained in \cite{dephilippismaggiCAPILLARI} asserts that free boundaries are regular outside closed sets of vanishing \(\H^{n-2}\)-measure. In this paper we improve upon this result by showing  $\H^{n-3}$-negligibility of singular sets, see Theorem \ref{thm main} below.

The ``interior part'' of this statement dates back to \cite{schoensimonalmgren}. The boundary case is addressed here by combining the set of ideas introduced in \cite{schoensimonalmgren} with the $\H^{n-2}$-negligibility we have obtained in \cite{dephilippismaggiCAPILLARI} (see, in particular, Lemma \ref{lemma II implica n-3} below).

We note that singular sets must necessarily be smaller than merely $\H^{n-3}$-negligible. Indeed, a general argument due to Almgren (and appeared in \cite[Lemma 5.1]{whitemodp}) implies that the set of \(s>0\) such that singular sets of minimizers of a given elliptic functional are \(\H^{s}\)-negligible is open. At the same time, the cone over ${\bf S}^1\times{\bf S}^1\subset\R^4$ minimizes a suitable elliptic anisotropic functional \cite{MorganClifford}. This example may lead to conjecture that singular sets of arbitrary anisotropic functionals have Hausdorff dimension at most $n-4$, although we are not aware of further evidence supporting this possibility.

The $\H^{n-3}$-negligibility of the singular set, although not optimal, has two interesting consequences. Firstly, and obviously, it implies everywhere regularity in \(\R^3\); secondly, it provides the  needed regularity in order to exploit second variation arguments in the study of geometric properties of minimizers; see for example \cite{SternZum2} and Lemma \ref{lemma potenziale stima II} below (actually \(\H^{n-3}\)-locally finiteness of the singular set would be enough for this, see for instance \cite[Section 4.7.2]{EvansGariepyBOOK}).

We now define the class of functionals and the notion of minimizers that we shall use.

\begin{definition}[Regular elliptic integrands] Given an open set $A\subset\R^n$, $\lambda\ge 1$ and $\ell\ge0$, we consider the family
  \(
  \X(A,\lambda,\ell)
  \)
  of  functions $\Phi:\cl(A)\times\R^n\to[0,\infty]$ such that $\Phi(x,\cdot)$ is convex and positively one-homogeneous on $\R^n$ with  $\Phi(x,\cdot)\in C^{2,1}(\mathbf S^{n-1})$ for every $x\in\cl(A)$, and such that the following properties hold for every $x\,,y\in \cl(A)$, $\nu,\nu'\in \mathbf S^{n-1}$, and $e\in\R^n$:
  \[
  \begin{gathered}
  \frac1\lambda\le\Phi(x,\nu)\le\lambda\,,
  \\
   |\Phi(x,\nu)-\Phi(y,\nu)|+|\nabla\Phi(x,\nu)-\nabla\Phi(y,\nu)|\le \ell\,|x-y|\,,
  \\
  |\nabla\Phi(x,\nu)|+\|\nabla^2\Phi(x,\nu)\|+\frac{\|\nabla^2\Phi(x,\nu)-\nabla^2\Phi(x,\nu')\|}{|\nu-\nu'|}\le\lambda\,,
  \\
  \end{gathered}
\]
and
   \begin{equation}
   \label{elliptic}
  \nabla^2\Phi(x,\nu)[e]\cdot e\ge\frac{\big|e-(e\cdot\nu )\nu\big|^2}{\lambda}\,.
  \end{equation}
  \end{definition}
  In the above definition  $\nabla\Phi$ and $\nabla^2\Phi$ stand for the gradient and Hessian of $\Phi$ in the $\nu$-variable, $\|L\|=\sup\{Le:|e|=1\}$ is the operator norm of a linear map $L:\R^n\to\R^n$, $L[e]$ is the action of $L$ on $e\in\R^n$, and $\cl(A)$ is the closure of $A$. We also set
 \[
 \X_*(\lambda)=\X(\R^n,\lambda, 0)\,,
 \]
 for the class of {\it regular autonomous elliptic integrand} (indeed, $\ell=0$ forces $\Phi(x,\nu)=\Phi(\nu)$). We shall regard $\X_*(\lambda)$ as a subset of $C^{2,1}(\mathbf S^{n-1})$ by the obvious identification of a one-homogeneous function with its trace on the sphere. With this identification it is immediate to check  that $\X_*(\lambda)$ is a compact subset with respect to uniform convergence on $\mathbf{S}^{n-1}$. Finally, if $\Phi\in\X(A,\lambda,\ell)$ and $E$ is a set of locally finite perimeter in $A$, then we set
 \[
 \PHI(E;G)=\int_{G\cap\pa^*E}\Phi(x,\nu_E(x))\,d\H^{n-1}(x)\in[0,\infty]\,,\qquad\forall G\subset A\,.
 \]
 Here $\pa^*E$ denotes the reduced boundary of $E$ in $A$ and $\nu_E$ is the measure-theoretic outer unit normal to $E$; see \cite[Chapter 15]{maggiBOOK}.

 \begin{definition}[Almost-minimizers]\label{def almost min} Let an open set $A$ and an open half-space $H$ in $\R^n$ be given (possibly \(H=\R^n\)), together with $r_0\in(0,\infty]$ and $\Lambda\ge0$. Given $\Phi\in\X(A\cap H,\lambda,\ell)$ and a set $E\subset H$ of locally finite perimeter in \(A\), one says that $E$ is a {\it $(\La,r_0)$-minimizer of $\PHI$ in $(A,H)$}, if
 \[
 \PHI(E;H\cap W)\le \PHI(F;H\cap W)+\Lambda\,|E\Delta F|\,,
 \]
 whenever $F\subset H$, $E\Delta F\cc W$, and $W\cc A$ is open with $\diam(W)<2r_0$;
 \begin{figure}
   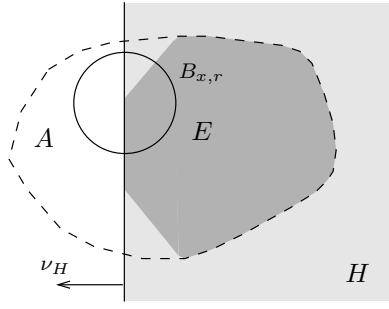\caption{{\small The situation in Definition \ref{def almost min}: roughly speaking, $E$ minimizes $\PHI$ with respect to perturbations $F$ which agree with $E$ on $H\cap\pa B_{x,r}$ and are allowed to freely move the boundary of $E$ close to $B_{x,r}\cap\pa H$. In other words, we impose a Dirichlet condition on $H\cap\pa B_{x,r}$ and a Neumann condition of $B_{x,r}\cap\pa H$. }}\label{fig min}
 \end{figure}
 see Figure \ref{fig min}.
 When $\La=0$, and $r_0=+\infty$, one simply says that $E$ is a {\it minimizer of $\PHI$ in $(A,H)$}.
 \end{definition}

\begin{remark} As proved in \cite[Lemma 6.1]{dephilippismaggiCAPILLARI}, up to local diffeomorphisms, minimizers of capillarity-type problems fall in the framework of Definition \ref{def almost min}. Other applications include relative isoperimetric problems in Riemannian and Finsler geometry.
\end{remark}

\begin{remark}\label{rmk: halfspace}
Since the class \(  \X(A\cap H,\lambda,\ell)\) is invariant by isometries of \(\R^n\) (in the sense that, if \(f(x)=x_0+R[x]\), \( R\in O(n)\), then
\(E\) is a   $(\La,r_0)$-minimizer of $\PHI$ in $(A,H)$ if and only if \(f(E)\)  is a  $(\La,r_0)$-minimizer of $\PHI^f$ in $(f(A),f(H))$ where
\(\Phi^f(x,\nu)=\Phi(f^{-1}(x),R^{-1} \nu)\) belongs to  \(\X(f(A)\cap f(H),\lambda,\ell)\), see \cite[Lemma 2.18]{dephilippismaggiCAPILLARI}) and we are interested in boundary regularity, in the sequel we can and do assume that \(H\) is a  fixed half-space with \(0\in \partial H\).

\end{remark}

Let now $E$ be a $(\La,r_0)$-minimizer of $\PHI$ in $(A,H)$ of some $\Phi\in\X(A\cap H,\lambda,\ell)$, and set
 \[
 M_A(E)=A\cap\cl(H\cap\pa E)\,.
 \]
 The {\it regular set $R_A(E)$ of $E$ in $A$} is defined by
 \[
 R_A(E)=\bigg\{x\in M_A(E):
  \begin{array}{l}
  \textrm{there exists \(r_x>0\) such that \(M_A(E)\cap B_{x,r_x}\)}
  \\
  \textrm{is a \(C^1\)-manifold with boundary contained in \(\pa H\)}
  \end{array}\bigg\}\,,
 \]
 while $\S_A(E)=M_A(E)\setminus R_A(E)$ is called the {\it singular set $\S_A(E)$ of $E$ in $A$}. In this way, $\S_A(E)$ is relatively  closed in  $A$. We shall also set
 \[
 R_G(E)=R_A(E)\cap G\,,\qquad \S_G(E)=\S_A(E)\cap G\,,\qquad\forall G\subset A\,.
 \]
 By combining the results of \cite{schoensimonalmgren} for the interior situation with the ones of \cite{dephilippismaggiCAPILLARI} for the boundary situation, one sees that $E\cap A$ is (equivalent to) an open set, that $A\cap\pa E\cap\pa H$ is a set of finite perimeter in $\pa H$, and that
 \begin{eqnarray}\label{n-3 interior}
   \H^{n-3}(\S_{A\cap H}(E))=0\,,&&\qquad\mbox{by \cite{schoensimonalmgren}}\,,
   \\\label{n-2 boundary}
   \H^{n-2}(\S_{A\cap\pa H}(E))=0\,,&&\qquad\mbox{by \cite{dephilippismaggiCAPILLARI}}\,,
 \end{eqnarray}
 with $\nabla\Phi(x,\nu_E)\cdot\nu_H=0$ at every $x\in R_{A\cap\pa H}(E)$. Moreover, one has a characterization of the regular and  singular sets in terms of the following notion of excess: given $x\in A$ and $r<\dist(x,\pa A)$ and denoting by \(B_{x,r}\) the open ball centered at \(x\) and with radius \(r\),  we define {\it spherical excess of $E$ at the point $x$, at scale $r$, relative to $H$} as
 \[
 {\bf exc}^H(E,x,r)=\inf\Big\{\frac1{r^{n-1}}\int_{B_{x,r}\cap H\cap\pa^*E}\frac{|\nu_E-\nu|^2}2\,d\H^{n-1}:\nu\in \mathbf S^{n-1}\Big\}\,.
 \]
 Then, for  positive constants $\e=\e(n,\lambda)$ and \(c=c(n,\lambda)\), we have that
 \begin{equation}\label{reg exc}
  {\bf exc}^H(E,x,r)<\e\quad \Longrightarrow \quad M_A(E)\cap B_{x, c r}\subset R_A(E)\,,
 \end{equation}
  see \cite[Theorem 3.1]{dephilippismaggiCAPILLARI}. In particular
 \begin{equation}\label{singular set char}
 \S_A(E)=\Big\{x\in M_A(E):\liminf_{r\to 0^+}{\bf exc}^H(E,x,r)\ge\e(n,\lambda)\Big\}\,\,,
 \end{equation}

 \begin{theorem}\label{thm main}
 If $\Phi\in\X(A,\lambda,\ell)$ and $E$ is a $(\La,r_0)$-minimizer of $\PHI$ in $(A,H)$, then
 \[
 \H^{n-3}(\S_{A\cap\pa H}(E))=0\,.
 \]
 \end{theorem}

 We now describe the proof of Theorem \ref{thm main}. First of all, by a blow-up argument, Theorem \ref{thm main} is seen to be equivalent to the following theorem.

 \begin{theorem}
   \label{lemma main}
 If $\Phi\in\X_*(\lambda)$, $B=B_{0,1}$, and $E$ is a minimizer of $\PHI$ in $(B,H)$, then
 \begin{equation}\label{zurigo1}
 \H^{n-3}(\S_{B\cap \pa H}(E))=0\,.
 \end{equation}
 \end{theorem}

We deduce Theorem \ref{lemma main} from the following two propositions, where we set
 \[
 \X_{**}(\lambda)=\Big\{\Phi\in\X_*(\lambda):\mbox{such that \eqref{zurigo1} holds true for every  $E$ is a minimizer of $\PHI$ in $(B,H)$}\Big\}\,.
 \]

 \begin{proposition}
  \label{lemma apertura}
  The set $\X_{**}(\lambda)$ is open in \(\X_*(\lambda)\) in the uniform convergence on \(\mathbf S^{n-1}\).
  \end{proposition}

  \begin{proposition}\label{lemma chiusura}
  The set $\X_{**}(\lambda)$ is closed \(\X_*(\lambda)\)  in the uniform convergence on \(\mathbf S^{n-1}\).
  \end{proposition}

  \begin{proof}
    [Proof of Theorem \ref{lemma main}] Obviously, \(\X_*(\lambda)\) is convex, thus connected. By \cite{gruter2} (or, alternatively, by \cite[Corollary 1.4]{dephilippismaggiCAPILLARI}) the isotropic functional \(\Phi(\nu)=|\nu|\) belongs to \(\X_{**}(\lambda)\) for all \(\lambda\ge 1\). Propositions \ref{lemma apertura} and \ref{lemma chiusura} thus imply \(\X_{**}(\lambda)=\X_*(\lambda)\).
  \end{proof}

  In section \ref{section proofs} we prove Propositions \ref{lemma apertura} and  \ref{lemma chiusura} and show that Theorem \ref{lemma main} implies Theorem \ref{thm main}. Second variation formulas used in these arguments are collected in appendix.

  We close this introduction by describing the main ideas behind the two key propositions. Proposition  \ref{lemma apertura} is based on the idea that, roughly speaking, for every \(s>0\) the map
\[
\Phi\mapsto \sup \Big\{ \H^{s}(\S_{B\cap\pa H}(E)):  \textrm{ \(E\) is a  minimizer of \(\PHI\) in \((B,H)\) }\Big\}
\]
is upper semi-continuous on $\X_*(\lambda)$ with respect to the uniform convergence on \(\mathbf S^{n-1}\). Concerning Proposition  \ref{lemma chiusura}, one starts by observing that, if $\Phi\in\X_*(\lambda)$, then $R_A(E)$ is a $C^{2}$-manifold with boundary.  Denoting by ${\rm II}_E$ the second fundamental form of $R_A(E)$, we set
  \begin{equation}
    \label{II E misura}
      |{\bf II}_E|^2(G)=\int_{G\cap R_E(A)}|{\rm II}_E|^2\,d\H^{n-1}\in[0,\infty]\,,\qquad\forall G\subset \R^n\,,
  \end{equation}
  where $|{\rm II}_E|^2$ is the squared Hilbert-Schmidt norm of the tensor ${\rm II}_E$, which equals the sum of the squared principal curvatures of $R_A(E)$. One then shows that $\Phi\in\X_{**}(\lambda)$ if and only if
  \[
  |{\bf II}_E|^2(B)\le C\qquad\textrm{ for every minimizer $E$ of $\PHI$ in $(B,H)$}\,,
  \]
  for some $C=C(n,\lambda)$, and hence concludes by proving that the map
  \[
  \Phi\mapsto \sup \Big\{   |{\bf II}_E|^2(B):  \textrm{  \(E\) is a  minimizers of \(\PHI\) in \((B,H)\)  }\Big\}
  \]
  is lower-semicontinuous on $\X_*(\lambda)$ with respect to the uniform convergence on $\mathbf S^{n-1}$.

  \medskip

  \noindent {\bf Acknowledgement}: FM was supported by the NSF Grant DMS-1265910.

  \section{Proofs}\label{section proofs}  Here and in the following we say that $E_h\to E$ in $A$ as $h\to\infty$ if $|(E_h\Delta E)\cap A|\to 0$ as $h\to\infty$, and that $E_h\to E$ locally in $A$ as $h\to\infty$ if, for every $K\cc A$, we have $E_h\to E$ in $K$ as $h\to\infty$. Moreover, we set set $I_\e(S)$ for the $\e$-neighborhood of $S\subset\R^n$.
%
We begin with a classical  lemma concerning convergence of minimizers and of singular sets, see for instance \cite[Lemma 28.14]{maggiBOOK}
  \begin{lemma}
    \label{lemma compactness}
    Let $\{\Phi_h\}_{h\in\N}\subset\X_*(\lambda)$ with $\Phi_h\to\Phi$ in $C^{0}(\mathbf S^{n-1})$ as $h\to\infty$, and let $\{E_h\}_{h\in\N}$ be such that $E_h$ is a $(\Lambda_h,r_h)$-minimizer of $\PHI_h$ in $(A,H)$ with $\Lambda_h\to\Lambda<\infty$ and $r_h\to r_0>0$ as $h\to\infty$. Then there exists a $(\Lambda,r_0)$-minimizer $E$ of $\PHI$ in $(A,H)$ such that, up to subsequences, $E_h\to E$ locally in $A$ as $h\to\infty$. Moreover, for every $\e>0$ and $K\cc A$ there exists $h_0>0$ such that
     \begin{equation}
      \label{half haus}
      \S_{K}(E_h)\subset I_\e(\S_{K}(E))\,,\qquad\forall  h\ge h_0\,.
    \end{equation}
    In particular,
    \begin{equation}
      \label{usc Hinfinito}
          \H^s_\infty(\S_{ K}(E))\ge\limsup_{h\to\infty}\H^s_\infty(\S_{ K}(E_h))\,,\qquad\forall s\in  [0,n]\,,
    \end{equation}
    where  \(\H^{s}_\infty\) is defined for every \(G\subset \R^n\)  as
  \[
  \H^{s}_\infty(G)=\inf \Big\{ \sum_{i\in \N } \om_s \Big(\frac{\diam(G_i)}2\Big)^s:  G\subset \bigcup_{i\in \N}G_i\,,\textrm{ \(G_i\) open}\Big\}\quad\mbox{with}\quad \om_s=\frac{\pi^{s/2}}{\int_0^\infty t^{s/2}\,e^{-t}\,dt}\,.
  \]
  \end{lemma}

  \begin{proof}
The local convergence in $A$ to a minimizer $E$ of $\PHI$ follows by \cite[Theorem 2.9]{dephilippismaggiCAPILLARI}. Since $\exc^H(E_h,x,r)\to\exc^H(E,x,r)$ for a.e. $r>0$ and for every $x\in A$ (cf. with \cite[Equation (3.10)]{dephilippismaggiCAPILLARI}) and by \eqref{reg exc} and \eqref{singular set char}, one proves \eqref{half haus}. Finally, if $\{G_i\}_{i\in \N}$ is an open covering of $\S_{ K}(E)$, then there exists $\e>0$ such that $\{G_i\}_{i\in \N}$ is a covering of $I_\e(\S_{K}(E))$, and thus of $\S_{K}(E_h)$ too, provided $h\ge h_0$: by minimizing on all the open coverings  we obtain \eqref{usc Hinfinito}.
  \end{proof}

  We now prove Proposition  \ref{lemma apertura} by using Lemma \ref{lemma compactness}. To this end we recall some properties of  \(\H^{s}_\infty\). First of all, \( \mathcal H_\infty^s\ge \mathcal H^s\), with
  \begin{equation}\label{zerossezero}
  \mathcal H^s(G)=0 \quad \textrm{if and only if}\quad  \mathcal H_\infty^s(G)=0\,.
  \end{equation}
  Moreover, for every \(G\subset \R^n\)  and  \(s\in [0,n]\) we have
  \begin{equation}\label{dens}
 \limsup_{r\to 0} \frac{\H_\infty ^{s}(G\cap B_{x,r})}{r^s}\ge c(s)>0\qquad \textrm{for \(\H^{s}\)-a.e.  \(x\in G\)\,,}
  \end{equation}
   see \cite[Theorem 3.26 (2)]{SimonLN}. We now set
    \[
  E^{x,r}=\frac{E-x}{r}\,,\qquad\forall x\in\R^n\,,r>0\,,
  \]
  and we notice that,  if $\Phi\in\X(A,\lambda,\ell)$ and $E$ is a $(\Lambda,r_0)$-minimizer of $\PHI$ in $(A,H)$, then $E^{x,r}$ is a $(\La\,r,r_0/r)$-minimizer of $\PHI^{x,r}$ in $(A^{x,r},H^{x,r})$, where
  \[
  \Phi^{x,r}(y,\nu)=\Phi(x+r\,y,\nu)\,,\qquad\forall y\in A^{x,r}\,, \nu\in \mathbf{S}^{n-1}\,.
  \]
  We shall also  frequently use the facts that if $x\in A\cap\pa H$ and \(0\in \partial H\) (see Remark \ref{rmk: halfspace}), then $H^{x,r}=H$ for every $r>0$ and $A^{x,r}$ eventually contains every compact set of $\R^n$ as $r\to 0$; and that if $\Phi\in\X_*(\lambda)$, then $\Phi^{x,r}=\Phi$.

  \begin{proof}[Proof of Proposition  \ref{lemma apertura}]
  Let $\Phi\in\X_{**}(\lambda)$ and assume there exists $\{\Phi_h\}_{h\in\N}\subset\X_*(\lambda)\setminus\X_{**}(\lambda)$ such that $\Phi_h\to\Phi$ in $C^{0}(\mathbf S^{n-1})$ as $h\to\infty$. In particular, for every $h\in\N$ there exists a minimizer $E_h$ of $\PHI_h$ in $(B,H)$ such that $\H^{n-3}(\S_{B\cap\pa H}(E_h))>0$. By \eqref{dens}  there exist $x_h\in \S_{B\cap\pa H}(E_h) $ and $r_h\to 0$ with
  \begin{equation}\label{limite}
  \frac{r_h} {\dist(x_h,\pa B)}\to 0\qquad \textrm{ as $h\to\infty$,}
  \end{equation}
  such that
  \[
  \H^{n-3}_\infty(\S_{B\cap\pa H}(E_h)\cap B_{x_h,r_h})\ge c(n)\,r_h^{n-3}\,.
  \]
  Let us set $F_h=(E_h)^{x_h,r_h}$. Then $F_h$ is a minimizer of $\PHI_h$ in $(B^{x_h,r_h},H)$ and
  \[
    \H^{n-3}_\infty(\S_{B\cap \pa H}(F_h))= \frac{\H_\infty ^{n-3}(\S_{B\cap\pa H}(E_h)\cap B_{x_h,r_h})}{r_h^{n-3}}\ge c(n)>0\,.
  \]
  By Lemma \ref{lemma compactness}, there exist a minimizer $F$ of $\PHI$ in $(\R^n,H)$ (since \(B^{x_h,r_h}\to \R^n\) by \eqref{limite}) such that $\H^{n-3}_\infty(\S_{B\cap\pa H}(F))>0$. By \eqref{zerossezero}, this contradicts the fact that $\Phi\in\X_{**}(\lambda)$.
  \end{proof}

  The same argument gives the following lemma.

  \begin{lemma}\label{lemma X figo}
  If $A$ is an open set, $\Phi\in\X_{**}(\lambda)$ and $E$ is a $(\Lambda,r_0)$-minimizer of $\PHI$ in $(A,H)$, then $\H^{n-3}(\S_{A\cap\pa H}(E))=0$.
  \end{lemma}

  \begin{proof}[Proof of Lemma \ref{lemma X figo}] If $E$ is a $(\Lambda,r_0)$-minimizer of $\PHI$ in $(A,H)$ with $\H^{n-3}(\S_{A\cap\pa H}(E))>0$, then by  arguing as in the proof of Proposition \ref{lemma apertura} we can find $r_h\to 0$ as $h\to\infty$ and $x\in\S_{A\cap\pa H}(E)$ such that
  \[
  \H^{n-3}_\infty(\S_{A\cap\pa H}(E)\cap B_{x,r_h})\ge c(n)\,r_h^{n-3}\,.
  \]
  Hence $E_h=E^{x,r_h}$ is $(\La\,r_h,r_0/r_h)$-minimizer of $\PHI$ in $(B^{x,r},H^{x,r})$. By Lemma \ref{lemma compactness} there exists a minimizer $F$ of $\PHI$ in $(\R^n,H)$ such that $\H^{n-3}_\infty(\S_{B\cap\pa H}(F))\ge c(n)$, against $\Phi\in\X_{**}(\lambda)$.
  \end{proof}

 We now come to the proof of Lemma \ref{lemma chiusura}. Given $\Phi_h\to\Phi$ and a minimizer $E$ of $\PHI$, we shall need to approximate $E$ by minimizers of \(\PHI_h\). This will be done by minimizing $\PHI_h$ plus a suitable lower order perturbation.

\begin{definition} Given $g\in L^\infty_{\loc}(A)$ one says that {\it $E$ is a minimizer of $\PHI+\int g$ on $(A,H)$} if $E\subset H$ is a set of locally finite perimeter in $A$, and
  \begin{equation}
    \label{minimality phi g}
      \PHI(E;W\cap H)+\int_{W\cap H\cap E}g(x)\,dx\le\PHI(F;W\cap H)+\int_{W\cap H\cap F}g(x)\,dx\,,
  \end{equation}
  whenever $F\subset H$ and $E\Delta F\cc W$ with $W\cc A$ open.
  \end{definition}
  Note that if $E$ is a minimizer of $\PHI+\int g$ on $(A,H)$, then  for every $A'\cc A$ one has that $E$ is a $(\Lambda,\infty)$-minimizer of $\PHI$ in $(A',H)$ with $\Lambda=\|g\|_{L^\infty(A')}$. In particular, $R_A(E)$ is always a $C^1$-manifold with boundary. Moreover, by exploiting the Euler-Lagrange equation associated to \eqref{minimality phi g} (more precisely, we use the second order elliptic PDE satisfied by the first derivatives of any function $u$ whose graph locally coincides with $R_A(E)$), one finds that, if in addition $g\in\Lip(\R^n)$, then $R_A(E)$ is actually a $C^{2,\a}$-manifold with boundary for every $\a<1$, and hence the second fundamental form ${\rm II}_E$ is a continuous function on $R_A(E)$. It thus makes sense to define a Borel measure $|{\bf II}_E|^2$ on $\R^n$ by setting
  \[
  |{\bf II}_E|^2=\,|{\rm II}_E|^2\,\H^{n-1}\llcorner R_A(E)\,,
  \]
  compare with \eqref{II E misura}. The continuity of ${\rm II}_E$ on $R_A(E)$ guarantees that $|{\bf II}_E|^2$ is a Radon measure on $A\setminus\S_A(E)$.

  \begin{lemma}
  \label{lemma sci II2}
  Let $\{\Phi_h\}_{h\in\N}\subset\X_*(\lambda)$ with $\Phi_h\to\Phi$ in $C^{0}(\mathbf S^{n-1})$ as $h\to\infty$, $\{g_h\}_{h\in\N}\subset\Lip(\R^n)$ with \(\Lip\, g_h\le C\) and \(g_h\to g\) locally uniformly on $\R^n$ as $h\to\infty$, and let $E_h$ (resp., $E$) be a minimizer of $\PHI_h+\int g_h$ (resp., $\PHI+\int g)$ on $(A,H)$, with $E_h\to E$ locally in $A$ as $h\to\infty$. Then,
  \begin{equation}
    \label{lsc}
      |{\bf II}_E|^2(A')\le\liminf_{h\to\infty}|{\bf II}_{E_h}|^2(A')\,,
  \end{equation}
  for every open set $A'\subset A$.
  \end{lemma}

  \begin{proof} The regularity, in particular \cite[Lemma 3.4]{dephilippismaggiCAPILLARI} theory ensures that if $x\in R_{A\cap H}(E)$, then there exist $h_x\in\N$, $r_x>0$ and $\nu_x\in \mathbf{S}^{n-1}$ such that, if we set
  \begin{eqnarray*}
  \C_x=x+\Big\{y\in\R^n:|y\cdot\nu_x|<r_x\,,\big|y-(y\cdot\nu_x)\nu_x\big|<r_x\Big\}\,,
  \\
  \D_x=x+\Big\{y\in \R^n:y\cdot\nu_x=0\,,\big|y-(y\cdot\nu_x)\nu_x\big|<r_x\Big\}\,,
  \end{eqnarray*}
  then $\C_x\cc A\cap H$ and there exist $u_h,u\in C^{2,\a}(\D_x)$ with $u_h\to u$ in $C^{2,\a}(\D_x)$ as $h\to\infty$ and
  \begin{eqnarray*}
    \C_x\cap\pa E&=&\C_x\cap R_A(E)=\Big\{z+u(z)\,\nu_x:z\in\D_x\Big\}\,,
    \\
    \C_x\cap\pa E_h&=&\C_x\cap R_A(E_h)=\Big\{z+u_h(z)\,\nu_x:z\in\D_x\Big\}\,,
  \end{eqnarray*}
  for every $h\ge h_x$. In particular, if $\vphi\in C^0(\C_x)$, then, as $h\to\infty$,
  \[
  \vphi(z,u_h)\,\sqrt{1+|\nabla u_h|^2}\,|{\rm II}_{E_h}(z+u_h\nu_x)|^2\to \vphi(z,u)\,\sqrt{1+|\nabla u|^2}\,|{\rm II}_{E}(z+u\nu_x)|^2\,,
  \]
  for every $z\in\D_x$, and, actually, locally uniformly on $z\in\D_x$. Thus, by the area formula for graphs one finds
  \[
  \int_{\R^n}\vphi\,d|{\bf II}_E|^2=\lim_{h\to\infty}
  \int_{\R^n}\vphi\,d|{\bf II}_{E_h}|^2\,,\qquad\forall \vphi\in C^0(\C_x)\,.
  \]
  By a covering argument we conclude that
  \begin{equation}
    \label{dai}
      \int_{\R^n}\vphi\,d|{\bf II}_E|^2=\lim_{h\to\infty}
  \int_{\R^n}\vphi\,d|{\bf II}_{E_h}|^2\,,\qquad\forall \vphi\in C^0_c((A\cap H)\setminus \S_A(E))\,.
  \end{equation}
  If now $A'\subset A$ is open, then by \eqref{dai},
  \[
  |{\bf II}_E|^2\Big((A'\cap H)\setminus\S_A(E)\Big)\le\liminf_{h\to\infty}|{\bf II}_{E_h}|^2\Big((A'\cap H)\setminus\S_A(E)\Big)\le
  \liminf_{h\to\infty}|{\bf II}_{E_h}|^2(A')\,.
  \]
  We deduce \eqref{lsc} as $|{\bf II}_E|^2(A\cap\pa H)=0$ and $|{\bf II}_E|^2(\S_A(E))=0$.
  \end{proof}

  We now exploit a second variation argument to show that the $\H^{n-3}$-negligibility of singular sets implies uniform $L^2$-estimates on second fundamental forms.

  \begin{lemma}\label{lemma potenziale stima II}
  Let $\Phi\in\X_{**}(\lambda)$, $g\in C^2(\R^n)$, $A$ be a bounded open set, and $E$ be a minimizer of $\PHI+\int g$ on $(A,H)$. Then,
  \[
        \frac{|{\bf II}_E|^2(B_{x,r})}{r^{n-3}}\le C_0(n,\lambda,\Lip(g))\,,\qquad\forall B_{x,2r}\cc A\,.
  \]
  \end{lemma}

  \begin{proof} By Lemma \ref{lemma variazione seconda} in the appendix, there exists a constant $C=C(n,\lambda,\Lip(g))$ such that
  \begin{eqnarray}
    \label{second variation inequality}
    \int_{R_A(E)}|{\rm II}_E|^2\,\zeta^2\,d\H^{n-1}\le C\,\int_{R_A(E)}|\nabla\zeta|^2+\zeta^2\,d\H^{n-1}\,,
  \end{eqnarray}
  whenever $\zeta\in C^1_c(A)$ with $\spt\zeta\cap \S_A(E)=\emptyset$. We shall now exploit $\Phi\in\X_{**}(\lambda)$ to deduce that \eqref{second variation inequality} holds true for every $\zeta\in C^1_c(A)$. To this end let us fix such a $\zeta\in C^1_c(A)$, and let us assume without loss of generality that $|\zeta|\le1$ on $\R^n$. Since $E$ is a $(\Lambda,\infty)$-minimizer of $\PHI$ in $(A,H)$, by Lemma \ref{lemma X figo} and by \eqref{n-3 interior} one has $\H^{n-3}(\S_A(E))=0$. In particular, given $\e>0$ we can find a countable cover $\{F_k\}_{k\in\N}$ of $\S_A(E)$ such that
  \begin{equation}
    \label{eps}
      \diam(F_k)<\e_k\,,\qquad \sum_{k\in\N}\e_k^{n-3}<\e\,.
  \end{equation}
  By \eqref{eps}, for every $k\in\N$ we choose $x_k\in F_k$ so that $F_k\subset B_{x_k,2\e_k}$. Since $\{B_{x_k,2\e_k}\}_{k\in\N}$ is an open covering of $\S_A(E)$, by compactness $\{B_{x_k,2\e_k}\}_{k=1}^N$ is an open covering of $\S_A(E)\cap\spt\zeta$ for some $N\in\N$, and thus of $I_\de(\S_A(E)\cap\spt\zeta)$ for some $\de>0$ such that $\de\to 0^+$ as $\e\to 0^+$. Correspondingly we consider $\psi_k\in C^1_c(B_{x_k,3\e_k};[0,1])$ such that
  \begin{equation}
    \label{eps2}
      \mbox{$\psi_k=1$ on $B_{x_k,2\e_k}$}\,,\qquad |\nabla\psi_k|\le\frac{2}{\e_k}\,,
  \end{equation}
  and set $\psi=\max\{\psi_k:1\le k\le N\}$. In this way,
  \[
  \psi=1\qquad\mbox{on $I_\de(\S_A(E)\cap\spt\zeta)$}\,.
  \]
  This implies that $\zeta_0=(1-\psi)\,\zeta$ is a Lipschitz function with $\spt\zeta_0\cap \S_A(E)=\emptyset$. By approximation, we can apply \eqref{second variation inequality} to $\zeta_0$ in order to find
  \begin{eqnarray}
    \label{second variation inequality 0}
    \int_{R_A(E)\setminus I_\de(\S_A(E))}|{\rm II}_E|^2\,\zeta^2\,d\H^{n-1}\le C\,\int_{R_A(E)}|\nabla\zeta|^2+|\nabla\psi|^2+\zeta^2\,d\H^{n-1}\,,
  \end{eqnarray}
  with $C=C(n,\lambda,\Lip(g))$. By the second conditions in \eqref{eps} and \eqref{eps2} we easily find
  \[
  \int_{R_A(E)}|\nabla\psi|^2\le \,\sum_{k=1}^N\,\int_{R_A(E)\cap B_{x_k,3\e_k}}|\nabla\psi_k|^2
  \le 4\,\sum_{k=1}^N\,\frac{P(E;B_{x_k,3\e_k})}{\e_k^2}\le C\,\sum_{k\in\N}\e_k^{n-3}<C\,\e\,,
  \]
  where we have used the upper density estimate $P(E;B_{x,r})\le C(n,\lambda)\,r^{n-1}$, see \cite[Equation (2.47)]{dephilippismaggiCAPILLARI}. By plugging this last estimate into \eqref{second variation inequality 0}, and then letting $\e\to 0^+$,  we conclude as desired that \eqref{second variation inequality} holds for every $\zeta\in C^1_c(A)$. Finally, for $B_{x,2r}\cc A$ and $\zeta\in C^1_c(B_{x,2r})$ with $\zeta=1$ on $B_{x,r}$ and $|\nabla\zeta|\le C/r$, \eqref{second variation inequality} gives
  \[
  |{\bf II}_E|^2(B_{x,r})\le C\,\frac{P(E;B_{x,r})}{r^2}\le C\,r^{n-3}\,,
  \]
  thanks again to the upper density estimate \cite[Equation (2.47)]{dephilippismaggiCAPILLARI}.
  \end{proof}

  We finally prove that if $|{\bf II}_E|^2$ is a finite measure, then the singular set is $\H^{n-3}$-negligible. We start with the following lemma.

  \begin{lemma}\label{lemma delta}
  There exists $\de=\de(n,\lambda)$ such that if $\Phi\in\X_*(\lambda)$, $E$ is a minimizer of $\PHI$ in $(B,H)$, $0\in\pa H$, and
  \[
  |{\bf II}_E|^2(B)\le\de\,,
  \]
  then $0\in R_E(B)$.
  \end{lemma}

  \begin{proof}
  We argue by contradiction. Let $\{\Phi_h\}_{h\in\N}\subset\X_*(\lambda)$ be such that for each $h\in\N$ there exists  a minimizer $E_h$ of $\PHI_h$ in $(B,H)$ with $|{\bf II}_{E_h}|^2(B)\to 0$ as $h\to\infty$ and $0\in\S_B(E_h)$ for every $h\in\N$. By the compactness of $\X_*(\lambda)$ and Lemma \ref{lemma compactness}, there exist $\Phi\in\X_*(\lambda)$ and $E$ a minimizer of $\PHI$ in $(B,H)$ such that, up to subsequences, $E_h\to E$ locally in $B$ as $h\to\infty$. Moreover, by \eqref{reg exc}, \eqref{singular set char} and the continuity of the excess, $0\in\S_B(E)$. By \eqref{half haus}, for every $\e>0$ and $r<1$ there exists $h_0$ such that $\S_{B_r}(E_h)\subset I_\e(\S_{B_r}(E))$ provided $h\ge h_0$. By Lemma \ref{lemma sci II2},
  \begin{eqnarray*}
  |{\bf II}_E|^2\Big(B\setminus\cl\Big(I_\e(\S_{B_r}(E))\Big)\Big)&\le&\liminf_{h\to\infty}|{\bf II}_{E_h}|^2\Big(B\setminus\cl\Big(I_\e(\S_{B_r}(E))\Big)\Big)
  \\
  &\le&
  \liminf_{h\to\infty}|{\bf II}_{E_h}|^2\Big(B\setminus\cl\Big(\S_{B_r}(E_h)\Big)\Big)=0\,.
  \end{eqnarray*}
  By the arbitrariness of $\e$ and $r$, $|{\bf II}_E|^2(B)=0$. We now show that this last fact implies the existence of {\it finitely} many hyperplanes $L_i$ such that
  \begin{equation}
    \label{cristino}
    M_{B_{1/2}}(E)\cap H=\bigcup_i\,L_i\cap B_{1/2}\cap H\,,\qquad L_i\cap L_j\cap B_{1/2}\cap H=\emptyset\qquad\forall i\ne j\,.
  \end{equation}
  Indeed, by $|{\bf II}_E|^2(B)=0$ we have that $R_B(E)$ is contained into the union of at most {\it countably} many hyperplanes $L_i$. Let us set $A_i=B\cap H\cap L_i$ and $R_i=R_{B\cap H}(E)\cap L_i$. We claim that
  \begin{equation}
    \label{claim}
      A_i\cap\pa_{L_i}R_i\subset\S_{B\cap H}(E)\,,
  \end{equation}
  where $\pa_{L_i}R_i$ denotes the boundary of $R_i$ as a subset of $L_i$. Indeed, $A_i\cap\pa_{L_i}R_i\subset M_{B}(E)\cap H$, so that if \eqref{claim} fails, then there exists $x\in A_i\cap\pa_{L_i}R_i\cap R_{B\cap H}(E)$. By using the local $C^1$-graphicality of $R_{B\cap H}(E)$ at $x$, we immediately see that $x$ belongs to the interior of $R_i$ seen as a subset of $L_i$, in contradiction with $x\in\pa_{L_i}R_i$. By \eqref{claim} and by \eqref{n-3 interior}, we find that $\H^{n-3}(A_i\cap\pa_{L_i}R_i)=0$, thus that $\H^{n-2}(A_i\cap\pa_{L_i}R_i)=0$. This implies that the distributional derivative of $1_{R_i}\in L^1_{\rm loc}(L_i)$ vanishes on the connected open set $A_i$: in other words, since $R_i\cap A_i\ne \emptyset$, it must be $R_i=A_i$. By the upper density estimate \cite[Equation (2.47)]{dephilippismaggiCAPILLARI}, there are finitely many hyperplanes $L_i$ such that $L_i\cap B_{1/2}\ne\emptyset$. This proves \eqref{cristino}. Since $0\in\S_{B\cap\pa H}(E)$, there must be $i\ne j$ such that $0\in L_i\cap L_j\cap\pa H$: but then, by \eqref{cristino}, $L_i\cap L_j\subset\S_{B\cap\pa H}(E)$, against \eqref{n-2 boundary}.
  \end{proof}

  \begin{lemma}\label{lemma II implica n-3}
  If $\Phi\in\X_*(\lambda)$, $E$ is a minimizer of $\PHI$ in $(B,H)$, and
  \[
  |{\bf II}_E|^2(B)<\infty\,,
  \]
  then $\H^{n-3}(\S_B(E))=0$.
  \end{lemma}

\begin{proof}
  By Lemma \ref{lemma delta} and by scaling
  \begin{equation}
    \label{to which}
      |{\bf II}_E|^2(B_{x,r})\ge \de\,r^{n-3}\,,\qquad\forall x\in \S_{B\cap\pa H}(E)\,,\quad r<\dist(x,\pa B)\,.
  \end{equation}
  We now prove that, if we fix $s\in(0,1)$ and set $\S_s=\S_{B_s\cap\pa H}(E)$ for the sake of brevity, then
  \begin{equation}
    \label{r intorno}
      \lim_{r\to 0^+}\frac{|I_r(\S_s)|}{r^3}=0\,.
  \end{equation}
  Let $r<1-s$ and let $\{x_i\}_{i=1}^{N(r)}\subset\S_s$ be such that $|x_i-x_j|>2r$ for every $i\ne j$ and $\inf_i|x-x_i|\le 2r$ for every $x\in\S_s$, i.e. $\{x_i\}_{i=1}^{N(r)}$ is a maximal $2r$-net on $\S_s$. In this way, $\{B_{x_i,r}\}_{i=1}^{N(r)}$ is a finite disjoint family of balls to which we can apply \eqref{to which}, and such that $I_r(\S_s)$ is covered by $B_{x_i,3\,r}$. Hence,
  \begin{eqnarray*}
    |I_r(\S_s)|\le 3^n\,N(r)\,r^n\le\frac{3^n\,r^3}\de\,\sum_{i=1}^{N(r)}|{\bf II}_E|^2(B_{x_i,r})\le\frac{3^n\,r^3}\de\,|{\bf II}_E|^2(I_r(\S_s))\,.
  \end{eqnarray*}
  Since, by assumption, $|{\bf II}_E|^2(B)<\infty$, we have
  \[
  \lim_{r\to 0^+}|{\bf II}_E|^2(I_r(\S_s))=|{\bf II}_E|^2(\S_s)=0\,,
  \]
  where in the last identity we have used the fact that $|{\bf II}_E|^2$ is concentrated on $R_B(E)$. This proves \eqref{r intorno}, which immediately implies $\H^{n-3}(\S_s)=0$ (note that this could be directly inferred by the previous proof, however \eqref{r intorno} provides a slightly stronger information). By the arbitrariness of $s$ we complete the proof.
\end{proof}

\begin{proof}[Proof of Proposition \ref{lemma chiusura}]
  Let us consider a sequence $\{\Phi_h\}_{h\in\N}\subset\X_{**}(\lambda)$ such that $\Phi_h\to\Phi$ in $C^0({\mathbf S}^{n-1})$ as $h\to\infty$ for some $\Phi\in\X_*(\lambda)$, and let $E$ be a minimizer of $\PHI$ in $(B,H)$. We fix $s\in(0,1)$ and consider the variational problems
  \begin{equation}
    \label{chiusura1}
    \inf\Big\{\PHI_h(F;H\cap B)+\int_{F}g_h(x)\,dx:F\subset H\,,F\Delta E\subset B_s\Big\}\,,
  \end{equation}
  where we have set
  \[
  g_h=\varphi_h\ast\Big(\dist(\cdot,E)-\dist(\cdot,E^c)\Big)\,,
  \]
  for a sequence of smooth mollifiers \(\{\varphi_h\}_h\); in particular, $g_h \in C^\infty(\R^n)$ with $\Lip g_h \le 1$ for every $h\in\N$. Let now $E_h$ be a minimizer in \eqref{chiusura1}: we claim that $E_h\to E$ in $B$ as $h\to\infty$. Indeed, by \cite[Theorem 2.9]{dephilippismaggiCAPILLARI} there exists $G\subset H$ such that, up to subsequences, $E_h\to G$ locally in $B_s$ as $h\to\infty$. By comparing $E_h$ with $E$ in \eqref{chiusura1}, by lower semicontinuity (see  \cite[Equation (2.64)]{dephilippismaggiCAPILLARI}), and setting $g=\dist(\cdot,E)-\dist(\cdot,E^c)$, one has
  \[
  \PHI(G;H\cap B)+\int_G g\le\liminf_{h\to\infty}\PHI_h(E_h;H\cap B)+\int_{E_h}g_h\le\PHI(E;H\cap B_s)+\int_E\,g\,.
  \]
  By minimality of $E$ (note that $G\Delta E\subset B_s\cc B$), $\PHI(E;H\cap B)\le \PHI(G;H\cap B)$, and thus
  \[
  0\ge \int_G g-\int_E g=\int_{G\setminus E}\dist(x,E)\,dx+\int_{E\setminus G}\dist(x,E^c)\,dx\,.
  \]
  In particular, $|E\Delta G|=0$, that is, $E_h\to E$ locally in $B_s$, thus in $B$ by $E_h\Delta E\subset B_s$, as $h\to\infty$.

  Since $E_h$ is a minimizer for $\PHI_h+\int g_h$ on $(B_s,H)$, by Lemma \ref{lemma potenziale stima II} (and  \(\Lip g_h \le 1\))  we find
  \[
      \frac{|{\bf II}_{E_h}|^2(B_{x,r})}{r^{n-3}}\le C(n,\lambda)\,,\qquad\forall B_{x,2r}\cc B_s\,.
  \]
  Hence,  by Lemma \ref{lemma sci II2}, one finds
  \[
        |{\bf II}_E|^2(B_{x,r})<\infty,\qquad\forall B_{x,2r}\cc B_s\,.
  \]
  By Lemma \ref{lemma II implica n-3} we have $\H^{n-3}(\S_{B_{x,r}\cap\pa H}(E))=0$ for every $B_{x,2r}\cc B_s$. By covering and by the arbitrariness of $s$ we find $\H^{n-3}(\S_{B\cap\pa H}(E))=0$. This shows that $\Phi\in\X_{**}(\lambda)$.
\end{proof}

As explained in the introduction, Propositions \ref{lemma apertura} and \ref{lemma chiusura} imply Theorem \ref{lemma main}. We finally deduce Theorem \ref{thm main} from this last result.

\begin{proof}
  [Proof of Theorem \ref{thm main}] The proof is essentially the same as that of Lemma \ref{lemma X figo}. Let us briefly sketch it:   assume by contradiction that there exist  constants \(\lambda\ge 1\), \(\ell\ge 0\), \(\Lambda \ge 0\), \(r_0>0\),  an open set \(A\),  \(\Phi\in \X(A\cap H, \lambda, \ell)\) and \(E\) a \((\Lambda ,r_0)\)-minimizer of $\PHI$ in $(A,H)$ such that
  \[
  \H^{n-3}(\Sigma_{A\cap \partial H} (E))>0\,.
  \]
  According to \eqref{dens} we can find \(x_0\in \Sigma_{A\cap \partial H}(E) \) and \(r_h\to 0\) as \(h\to \infty\) such that
    \begin{equation}\label{sempre lei}
  \H^{n-3}_\infty(\Sigma_{A\cap \partial H} (E)\cap B_{x_0,r_h})>c(n) r_h^{n-3}.
  \end{equation}
Let us set  \(F_h=E^{x_0,r_h}\) and notice that \(F_h\) are \((\Lambda r_h ,r_0/r_h)\)-minimizer of $\PHI_h$ in $(A^{x_0,r_h},H)$ where \(\Phi_h(x,\nu)=\Phi(x_0+r_h x,\nu) \in \X(A^{x_0,r_h} \cap H, \lambda, \ell r_h )\).
According to Lemma \ref{lemma compactness} and arguing as in the proof of Lemma  \ref{lemma X figo} one finds  \(E_\infty\) a minimizer of \(\PHI_\infty\) in \((\R^n,H)\) where \(\Phi_\infty(\nu)=\Phi(x_0,\nu)\in \X_{*}(\lambda)\). However,  by  \eqref{sempre lei}, \eqref{usc Hinfinito} and \eqref{zerossezero}, we find $\H^{n-3}(\Sigma_{B\cap \partial H} (E_\infty))>0$,  a contradiction to Theorem \ref{lemma main}.
\end{proof}

\appendix

\section{First and second variations of anisotropic functionals} Lemma \ref{lemma potenziale stima II} relies on the second variation formulas for  anisotropic functionals. For the reader's convenience, and since this kind of computation is not so easily accessible in the literature, we include a derivation of these formulas.

We consider an open set with smooth boundary $\Omega$ in $\R^n$, a bounded open set $A$ with $A\cap\Omega\ne\emptyset$, and a set $E\subset\Om$  of finite perimeter in $A$. Given $\Phi\in\X_*(\lambda)$ and $g\in C^2(\R^n)$, we compute the first and second variation of
\[
(\PHI+\smallint g)(f_t(E))=\int_{A\cap\Om\cap\pa^*f_t(E)}\Phi(\nu_{f_t(E)})\,d\H^{n-1}+\int_{A\cap f_t(E)} g\,,
\]
where \(\{f_t\}_{|t|\le \e_0}\) is such that:
\begin{itemize}
\item[(i)]  $(x,t)\mapsto f_t(x)$ of class $C^1(\Om\times(-\e_0,\e_0);\Om)$ with $f_0=\Id$, $f_t(\Om)=\Om$ for every $|t|<\e_0$, and $t\in(-\e_0,\e_0)\mapsto f_t(x)$ of class $C^3((-\e_0,\e_0);\Om)$ uniformly with respect to $x\in\Om$;
\item[(ii)] \(\spt(f_t-\Id)\cc A\).
\end{itemize}
These conditions imply that
\begin{equation}\label{tangente1}
\frac {d}{dt} f_t(x)\cdot \nu_\Om(f_t(x))=0\,,\qquad  x\in \partial \Om\cap A\,,\quad |t|<\e_0\,.
\end{equation}
We also notice that, if we define $T,Z\in C^1_c(\Om;\R^n)$ by setting
\begin{equation}
  \label{T e Z}
  T(x)=\frac{d}{dt}\Big|_{t=0}f(x)\qquad \textrm{and}\qquad  Z(x)=\frac{d^2}{dt^2}\Big|_{t=0}f_t(x)\,,
\end{equation}
then we have, uniformly on $x\in\R^n$ as $t\to 0^+$,
\begin{equation}\label{taylorf}
f_t=\Id+tT+\frac{t^2}{2}Z+O(t^3)\,.
\end{equation}
By \eqref{tangente1} we find
\begin{equation}\label{tangente}
T\cdot \nu_\Om=0\,,\qquad \forall\, x\in \partial \Om\,.
\end{equation}
By differentiating \eqref{tangente1} with respect to \(t\) we obtain that
\begin{equation}\label{tangente2}
Z\cdot \nu_\Om=-T\cdot\II_{\Om}[T]\,,\qquad\forall x\in\pa\Om\,,
\end{equation}
where  \(\II_\Om:T_x\pa\Om\to T_x\pa\Om\) is the second fundamental form of \(\partial \Om\). (Note that $T(x)$ is a tangent vector to $\pa\Om$ at $x\in\pa\Om$ exactly by \eqref{tangente}.) We now recall two basic facts. Lemma \ref{lemma2} is consequence of the classical area formula, see for example \cite[Proposition 17.1]{maggiBOOK}, while Lemma \ref{lemma3} is a standard Taylor expansion, see \cite[Lemma 17.4]{maggiBOOK}.



\begin{lemma}\label{lemma2} If $f:\R^n\to\R^n$ is a Lipschitz diffeomorphism with $\det(\nabla f)>0$ on $\R^n$, then   $f(E)$ is a set of finite perimeter in $f(A)$, with $f(\pa^*E)=_{\H^{n-1}}\pa^*(f(E))$ and
\[
\nu_{f(E)}(f(x))=\frac{\cof(\nabla f(x))[\nu_E(x)]}{|\cof(\nabla f(x))[\nu_E(x)]|}\,,\qquad\mbox{for $\H^{n-1}$-a.e. $x\in\pa^*f(E)$}\,,
\]
where for any invertible  linear map  $L:\R^n\to\R^n$ one defines $\cof L= (\det L)\,(L^{-1})^*$. Moreover, for every  $G\subset A$, one has
\begin{equation}
\label{change of variables cofattore}
  \int_{f(G\cap\pa^*E)}\Phi(\nu_{f(E)}(y))\,d\H^{n-1}(y)=\int_{G\cap\pa^*E}\Phi\big(\cof(\nabla f(x))\,[\nu_{E}(x)]\big)\,d\H^{n-1}(x)\,.
\end{equation}
\end{lemma}

\begin{lemma}\label{lemma3}
If \(X,Y:\R^n\to\R^n\) are linear maps, then
\begin{equation}\label{det}
\det \Big(\Id+tX+\frac{t^2}{2} Y+O(t^3)\Big)=1+t\,\tr X+\frac{t^2}{2}\Big((\tr X)^2-\tr (X^2)+\tr Y\Big)+O(t^3)\,,
\end{equation}
\[
\Big(\Id+tX+\frac {t^2}{2} Y+O(t^3)\Big)^{-1}=\Id-t X+\frac{ t^2}{2} \Big(2 X^2-Y\Big)+O(t^3)\,,
\]
and thus
\[
\begin{split}
 \cof\Big(\Id+tX&+\frac{t^2}2\,Y+O(t^3)\Big)\\
 &=\Id+t\Big(\tr(X)\Id-X^*\Big)\label{cof}\\
 &+\frac{t^2}2\Big[\Big(\tr(X)^2-\tr(X^2)+\tr(Y)\Big)\Id+2\,(X^*)^2-2\tr(X)\,X^*-Y^*\Big]+O(t^3)\,.
\end{split}
\]
\end{lemma}
We are now ready to compute the first and second variation of $\PHI+\smallint g$.

\begin{lemma}  If \(g\in C^2(A)\), then
\begin{equation}\label{firstderivativepot}
\frac{d}{dt}\Big|_{t=0}\int_{A\cap f_t(E)} g=\int_{A\cap\Om\cap\partial^* E} g\,( T \cdot \nu_E) \,d\H^{n-1}\,,
\end{equation}
and
\begin{equation}\label{secondderivativepot}
\begin{split}
\frac{d^2}{dt^2}\Big|_{t=0}\int_{A\cap f_t(E)} g&=\int_{A\cap\Om\cap\partial^* E} g\, (Z\cdot \nu_E)\, d\H^{n-1}
\\
&+\int_{A\cap\Om\cap\partial^* E} \Div (g\, T)\, (T\cdot \nu_E)-g\, (\nabla T[T]\cdot \nu_E)\, d\H^{n-1}\,.
\end{split}
\end{equation}
\end{lemma}

\begin{proof} {\it Step one}: We notice the validity of the following formula: if $S\in C^1_c(A;\R^n)$ and \(E\subset\Omega\), then
\begin{eqnarray*}
  &&\int_{A\cap E}\,g\big[(\Div S)^2-\tr (\nabla S)^2\big]+2\,\Div S\, \nabla g \cdot S+\nabla^2 g[S]\cdot S
  \\
  &&\hspace{3cm}=\int_{\Om\cap A\cap\pa^*E} \Div (g\,S) (S\cdot\nu_E)-g \,\nabla S[S]\cdot\nu_E\,d\H^{n-1}
  \\
  &&\hspace{3.2cm}+\int_{A\cap\partial \Omega \cap \partial ^* E} \Div (g\,S) (S\cdot\nu_\Om)-g \,\nabla S[S]\cdot\nu_\Om\,d\H^{n-1}\,,
\end{eqnarray*}
where $E^{(1)}$ is the set of points of density one of $E$. Indeed, if $S\in C^2_c(A;\R^n)$, then the assertion follow by the divergence theorem and by the identity
\[
\begin{split}
g\big[(\Div S)^2-\tr (\nabla S)^2\big]&+2\,\Div S\, \nabla g \cdot S+\nabla^2 g[ S]\cdot S
=\Div(\Div (gS) S)-\Div(g \,\nabla S[S])\,.
\end{split}
\]
The case when $S\in C^1_c(A;\R^n)$ is then obtained by approximation.

\medskip

\noindent {\it Step two}: Since \(f_t(A)=A\), we find \(f_t(E)\cap A=f_t(E\cap A)\). Hence by the area formula,
\[
\int_{A\cap f_t(E)} g(y)\,dy=\int_{A\cap E} g(f_t(x))\det \nabla f_t(x) dx\,.
\]
By  \eqref{taylorf}, by \eqref{det} and by the Taylor expansion of $g$ we get
\[
\begin{split}
&\int_{A\cap f_t(E)}g(y)\,dy=\int_{A\cap E} g+t\int _{A\cap E} \nabla g \cdot T+g\,\Div T
\\
&+\frac{t^2}{2}\int _{A\cap E} g\big[\Div Z+(\Div T)^2-\tr (\nabla T)^2]+2\,\Div T\, \nabla g \cdot T+\nabla^2 g [T]\cdot T+\nabla g\cdot Z+O(t^3)\,.
\end{split}
\]
Inasmuch, $\Div(g\, T)=\nabla g \cdot T+g\,\Div T$ and $\Div(g\, Z)=\nabla g \cdot Z+g\,\Div Z$, by step one and by \eqref{tangente}, one finds \eqref{firstderivativepot} and
\[
\begin{split}
\frac{d^2}{dt^2}\Big|_{t=0}\int_{A\cap f_t(E)} g&=\int_{A\cap\Om\cap\pa^* E} g\, (Z\cdot \nu_E)\, d\H^{n-1}
\\
&+\int_{A\cap\Om\cap\partial^* E} \Div (g\, T)\, (T\cdot \nu_E)\, d\H^{n-1}\\
&-\int_{A\cap\Om\cap\partial^* E } g\, (\nabla T[T]\cdot \nu_E)\, d\H^{n-1}\\
&+\int_{A\cap\pa\Om\cap \partial^* E} g\, \big(Z\cdot \nu_\Om-\nabla T[T]\cdot \nu_\Om\big)\, d\H^{n-1}\,.
\end{split}
\]
We now complete the proof of \eqref{secondderivativepot} by showing that $\nabla T[T]\cdot \nu_\Om=Z\cdot\nu_\Om$. Indeed, by differentiating \eqref{tangente} along $T$ one finds $0=\nabla T[T]\cdot\nu_\Om+T\cdot\,\II_\Om[T]$, and then conclude by \eqref{tangente2}.
\end{proof}

\begin{lemma}\label{tens}We have
\begin{equation}\label{firstderivativetens}
\frac{d}{dt}\Big|_{t=0}\int_{A\cap\Om\cap\partial^*f_t(E)} \Phi(\nu_{f_t(E)})\,d\H^{n-1}=\int_{A\cap\Om\cap\partial^*E}\Phi(\nu_E)\Div T-\nabla T^*[\nu_E ]\cdot \nabla \Phi(\nu_E)\,d\H^{n-1}\,,
\end{equation}
and
\begin{equation}\label{secondderivativetens}
\begin{split}
\frac{d^2}{dt^2}\Big|_{t=0}\int_{A\cap\Om\cap\partial^*f_t(E)}& \Phi(\nu_{f_t(E)})\,d\H^{n-1}=\int_{A\cap\Om\cap\partial^*E}\Phi(\nu_E)\Div Z-\nabla Z^* [\nu_E] \cdot \nabla \Phi(\nu_E)\,d\H^{n-1}
\\
&+\int_{A\cap\Om\cap\partial^*E}\Phi(\nu_E)\big\{(\Div T)^2-\tr (\nabla T)^2\big\}\,d\H^{n-1}
\\
&+2\,\int_{A\cap\Om\cap\partial^*E}(\nabla T^*)^2[\nu_E]\cdot \nabla \Phi(\nu_E)-\Div T\nabla T^*[\nu_E]\cdot \nabla \Phi(\nu_E)\,d\H^{n-1}
\\
&+\int_{A\cap\Om\cap\pa^*E}\nabla^2 \Phi(\nu_E)\big[ \nabla T^*[\nu_E]\big]\cdot \nabla T^*[\nu_E]\,d\H^{n-1}\,.
\end{split}
\end{equation}
\end{lemma}

\begin{proof} By  \eqref{taylorf}, Lemma \ref{lemma3}, and by the Taylor expansion of $\Phi$ at $\nu_E$, we get
\[
\begin{split}
\Phi\big(\cof(\nabla f_t(x))[\nu_E]\big)&=\Phi(\nu_E)+t\Big\{\Phi(\nu_E)\Div T-\nabla T^*[\nu_E ]\cdot \nabla \Phi(\nu_E)\Big\}
\\
&+\frac{t^2}{2}\Big\{ \Phi(\nu_E)\Div Z-\nabla Z^* [\nu_E] \cdot \nabla \Phi(\nu_E)
\\
&\qquad+\Phi(\nu_E)\big\{(\Div T)^2-\tr (\nabla T)^2\big\}-2\Div T\,\nabla T^*[\nu_E]\cdot \nabla \Phi(\nu_E)
\\
&\qquad+2(\nabla T^*)^2[\nu_E]\cdot \nabla \Phi(\nu_E)+\nabla^2 \Phi(\nu_E)\big[ \nabla T^*[\nu_E]\big]\cdot \nabla T^*[\nu_E]\, \Big\}+O(t^3)\,,
\end{split}
\]
where we have also used $\Phi(\nu_E)=\nabla\Phi(\nu_E)\cdot\nu_E$ and $\nabla^2\Phi(\nu_E)[\nu_E]=0$. By \eqref{change of variables cofattore} and by $f_t(A)=A$ we find \eqref{firstderivativetens} and \eqref{secondderivativetens}.
\end{proof}

We now come to the lemma that was used in the proof of Lemma \ref{lemma potenziale stima II}. In the following we define \(\II^\Phi_E\) by setting
\[
  \II^\Phi_E(x) =\nabla^2\Phi(\nu_E(x))\,\II_E(x)\qquad \forall \, x\in R_A(E)\,.
  \]
 Note that, by one-homogeneity of \(\Phi\), \(\nabla^2 \Phi(\nu_E)[\nu_E]=0\); therefore, by symmetry of $\nabla^2\Phi(\nu_E)$, the tensor \(\II^\Phi_E(x)\) is a well defined operator from \(T_{x} R_A(E)\) into itself.

\begin{lemma}\label{lemma variazione seconda}
  Let $\Phi\in\X_{*}(\lambda)$, $g\in C^2(\R^n)$, $A$ be a bounded open set, \(H\) an open  half-space and $E$ be a minimizer of $\PHI+\int g$ on $(A,H)$. Then
     \begin{equation}
    \label{second variation}
    \begin{split}
   \int_{R_A(E)}\zeta^2 \Phi(\nu_E)&  \tr [(\II^\Phi_E)^2]\,d\H^{n-1}\\
   &\le     \int_{R_A(E)}\Phi(\nu_E)^2\,\nabla^2\Phi(\nu_E)[\nabla\zeta]\cdot \nabla\zeta+\zeta^2\Phi(\nu_E) (\nabla g\cdot\nabla\Phi(\nu_E)) \,d\H^{n-1}\,,
  \end{split}
  \end{equation}
  for every $\zeta\in C^1_c(A)$ with $\spt\zeta\cap \S_A(E)=\emptyset$.  Moreover, there exists a constant $C=C(n,\lambda,\Lip(g))$ such that
  \begin{eqnarray}
    \label{second variation inequality appendix}
    \int_{R_A(E)}|{\rm II}_E|^2\,\zeta^2\,d\H^{n-1}\le C\,\int_{R_A(E)}|\nabla\zeta|^2+\zeta^2\,d\H^{n-1}\,,
  \end{eqnarray}
  whenever $\zeta\in C^1_c(A)$ with $\spt\zeta\cap \S_A(E)=\emptyset$.
 \end{lemma}

\begin{proof} As proved in \cite[Section 2.4]{dephilippismaggiCAPILLARI} we have
\[
\nabla \Phi(\nu_E(x))\cdot \nu_H=0\qquad \forall\,x\in R_A(E)\cap \pa H\,.
\]
If $\zeta\in C^1_c(A\setminus\S_A(E))$, then there exists $N\in C^1(\R^n;\R^n)$ such that
\begin{eqnarray}\label{N1}
  N=\nu_E\,,&&\qquad\mbox{on $R_A(E)\cap\spt\,\zeta$}\,,
  \\\label{N2}
  \nabla\Phi(N)\cdot\nu_H=0\,,&&\qquad\mbox{on $R_A(E)\cap\pa H\cap\spt\,\zeta$}\,.
\end{eqnarray}
 We set  $T=\zeta\,\nabla\Phi(N)\in C_{c}^1(A;\R^n)$ and we note that, by \eqref{N2}, \(f_t(x)=x+tT(x)\) defines a family of admissible variations for \(|t|\le \e_0\) and $\e_0$ suitably small. Since $f_t$ is affine in $t$, by \eqref{T e Z}, one has $Z=0$. In particular, by Lemma \ref{lemma2}, Lemma \ref{lemma3}, and by minimality of $E$,
 \begin{eqnarray}\label{prima}
 &&0=\frac{d}{dt}\Big|_{t=0}(\Phi+\smallint g)(f_t(E))=\int_{A\cap H\cap\pa^*E}\,g\,(T\cdot\nu_E)+\Phi\,\Div T-(\nabla T)^*[\nu_E]\cdot\nabla\Phi\,d\H^{n-1}\,,\hspace{1cm}
 \\\nonumber
 &&0\le\frac{d^2}{dt^2}\Big|_{t=0}(\Phi+\smallint g)(f_t(E))=\int_{A\cap H\cap\pa^*E}\,\Gamma_1+\Gamma_2+\Gamma_3+\Gamma_4\,d\H^{n-1}\,,
 \end{eqnarray}
 where, setting for simplicity \(\Phi=\Phi(\nu_E)\), \(\nabla \Phi=\nabla \Phi(\nu_E)\), and \(\nabla^2 \Phi=\nabla^2 \Phi(\nu_E)\), one has
 \begin{eqnarray*}
   \Gamma_1&=&\Div(g\,T)\,(T\cdot\nu_E)-g\,\nabla T[T]\cdot\nu_E\,,
   \\
   \Gamma_2&=&\big((\Div T)^2-\tr((\nabla T)^2\big)\,\Phi\,,
   \\
   \Gamma_3&=&2\,\Big((\nabla T^*)^2[\nu_E]\cdot \nabla \Phi-\Div T\,\nabla T^*[\nu_E]\cdot \nabla \Phi\Big)\,,
   \\
   \Gamma_4&=&\,\nabla^2 \Phi\,\big[ \nabla T^*[\nu_E]\big]\cdot \nabla T^*[\nu_E]\,.
 \end{eqnarray*}
 We start by noticing that \eqref{N1} gives
 \[
 \nabla N(x)=\II_E(x)+a(x)\otimes\nu_E(x)\qquad \forall\,x\in R_A(E)\cap \spt \zeta\,,
 \]
 where $\II_E(x)$ is extended to be zero on $(T_x R_A(E))^\perp$ and \(a: R_A(E)\to \R^n\) is a continuous vector field. Hence
 \[
    \nabla T=\nabla\Phi\otimes\nabla\zeta+\zeta\,\II_E^{\Phi}+\zeta\,\nabla^2\Phi[a]\otimes\nu_E\,,\qquad\mbox{on $R_A(E)$}\,.
 \]
 By $\nabla^2\Phi\,[\nu_E]=0$ and the symmetry of $\nabla^2\Phi$ one finds $\tr(\nabla^2\Phi[a]\otimes\nu_E)=0$, so that
   \begin{eqnarray}\label{divT}
    \Div T=\nabla\Phi\cdot\nabla\zeta+\zeta\,H^\Phi_E\,,\qquad\mbox{on $R_A(E)$}\,,
 \end{eqnarray}
 where we have set
 \[
 H_E^\Phi=\tr (\II^{\Phi}_E)=\tr (\nabla^2 \Phi \,\II_E)\,.
 \]
 Moreover, by $\nabla\Phi\cdot\nu_E=\Phi$ and again by $\nabla^2\Phi\,[\nu_E]=0$ we find $(\nabla T)^*[\nu_E]=\Phi\,\nabla\zeta$ and $T\cdot\nu_E=\zeta\Phi$, so that \eqref{prima} gives
 \[
 0=\int_{A\cap H\cap\pa^*E}\,\big(g+H^\Phi_E)\,\Phi\,\zeta\,d\H^{n-1}\,.
 \]
 The validity of this condition for every $\zeta\in C^1_c(A\setminus\S_A(E))$ gives the well-know stationarity condition
 \begin{equation}
    \label{euler}
    H^\Phi_E +g=0\,,\qquad\forall\,x\in R_A(E)\,.
 \end{equation}
 We now compute $\Gamma_1$. By $\nabla\Phi\cdot\nu_E=\Phi$, we find
 \[
 \nabla T[T]=\zeta\,(\nabla\zeta\cdot\nabla\Phi)\,\nabla\Phi+\zeta^2\,\II_E^\Phi[\nabla\Phi]+\zeta^2\,\Phi\,\nabla^2\Phi[a]\,,
 \]
 so that, by $\II_E^\Phi[\nabla\Phi]\cdot\nu_E=0$ and by $\nabla^2\Phi[a]\cdot\nu_E=0$ (which follow by the symmetry of $\nabla^2\Phi$ and by $\nabla^2\Phi[\nu]=0$), we find
 \[
 \nabla T[T]\cdot\nu_E=\zeta\,\Phi\,(\nabla\zeta\cdot\nabla\Phi)\,.
 \]
 By \eqref{divT}, \eqref{euler} and a simple computation one gets
 \begin{eqnarray*}
   \Gamma_1=\Big((\nabla\Phi\cdot\nabla g)+g\,H_E^\Phi\Big)\zeta^2\,\Phi=\Big((\nabla\Phi\cdot\nabla g)-(H_E^\Phi)^2\Big)\zeta^2\,\Phi\,.
 \end{eqnarray*}
 We now start computing $\Gamma_2$. By \eqref{divT} we have
 \[
 (\Div T)^2=(\nabla\Phi\cdot\nabla\zeta)^2+\zeta^2\,(H_E^\Phi)^2+2\,\zeta\,H_E^\Phi\,(\nabla\Phi\cdot\nabla\zeta)\,;
 \]
 at the same time, writing $\nabla T=X+Y$ where $X=\nabla\Phi\otimes\nabla\zeta+\zeta\II_E^\Phi$ and $Y=\zeta \nabla^2\Phi[a]\otimes\nu_E$, and noticing that $Y^2=0$, while
 \begin{eqnarray*}
 \tr(Y\,X)=\tr(X\,Y)&=&\tr\Big(\zeta(\nabla\zeta\cdot\nabla^2\Phi[a])\,\nabla\Phi\otimes\nu_E+\zeta^2\,\II_E^\Phi\,\nabla^2\Phi[a]\otimes\nu_E\Big)
 \\
 &=&
 \zeta(\nabla\zeta\cdot\nabla^2\Phi[a])\,\Phi\,,
 \\
 X^2&=&(\nabla\zeta\cdot\nabla\Phi)\nabla\Phi\otimes\nabla\zeta+\zeta^2(\II_E^\Phi)^2+
    \zeta\,\II_E^\Phi[\nabla\Phi]\otimes \nabla\zeta+\zeta\,\nabla\Phi\otimes(\II_E^\Phi)^*[\nabla\zeta]\,,
 \end{eqnarray*}
 we find that,
 \begin{eqnarray*}
      \tr((\nabla T)^2)&=&
      (\nabla\zeta\cdot\nabla\Phi)^2+\zeta^2\tr[(\II_E^\Phi)^2]+
    2\,\zeta\,(\nabla\zeta\cdot\II_\Phi[\nabla\Phi])+2\,(\nabla\zeta\cdot\nabla^2\Phi[a])\,\Phi\,.
 \end{eqnarray*}
 Hence,
 \begin{eqnarray*}
 \Gamma_2&=&\zeta^2\,(H_E^\Phi)^2\Phi+2\zeta(\nabla\zeta\cdot\nabla\Phi)H_E^\Phi\,\Phi-\zeta^2\tr[(\II_E^\Phi)^2]\,\Phi
 \\
 &&-2\zeta\,(\nabla\zeta\cdot\II_E^\Phi[\nabla\Phi])\Phi-2\,(\nabla\zeta\cdot\nabla^2\Phi[a])\,\Phi^2\,.
  \end{eqnarray*}
  We now compute $\Gamma_3$. By \eqref{divT} and $(\nabla T)^*[\nu_E]=\Phi\,\nabla\zeta$, we find
  \begin{eqnarray*}
    \Div T\,\nabla T^*[\nu_E]\cdot\nabla\Phi=
    (\nabla\zeta\cdot\nabla\Phi)^2\Phi+\zeta H_E^\Phi\,(\nabla\zeta\cdot\nabla\Phi)\,\Phi\,.
  \end{eqnarray*}
  At the same time, writing $\nabla T=X+Y$ with $X$ and $Y$ as above, we find
  \begin{eqnarray*}
     (X^*)^2&=&(\nabla\zeta\cdot\nabla\Phi)\nabla\zeta\otimes\nabla\Phi+\zeta^2(\II_E^\Phi)^2+
    \zeta\,\nabla\zeta\otimes\II_E^\Phi[\nabla\Phi]+\zeta\,(\II_E^\Phi)^*[\nabla\zeta]\otimes \nabla\Phi\,,
    \\
    Y^*X^*&=&\zeta(\nabla\zeta\cdot\nabla^2\Phi[a])\,\nu_E\otimes\nabla\Phi+\zeta^2\,(\nu_E\otimes \nabla^2\Phi[a])\,\II_E^\Phi
    \\
    X^*Y^*&=&\zeta\,\Phi\,\nabla\zeta\otimes\nabla^2\Phi[a]\,.
  \end{eqnarray*}
  By taking into account that $(Y^*)^2=0$ (as $Y^2=0$) and by exploiting once more that $\nabla^2\Phi[\nu_E]=0$ and $\II_E^\Phi[\nu_E]=0$, we find that
  \[
  [(\nabla T)^*]^2[\nu_E]=(\nabla\zeta\cdot\nabla\Phi)\,\Phi\,\nabla\zeta+\zeta\,\Phi\,(\II_E^\Phi)^*[\nabla\zeta]
  +\zeta(\nabla\zeta\cdot\nabla^2\Phi[a])\,\Phi\,\nu_E\,,
  \]
  so that
  \[
  [(\nabla T)^*]^2[\nu_E]\cdot\nabla\Phi=(\nabla\zeta\cdot\nabla\Phi)^2\,\Phi+\zeta\,\nabla\zeta\cdot\,\II_E^\Phi[\nabla\Phi]\,\Phi
  +\zeta(\nabla\zeta\cdot\nabla^2\Phi[a])\,\Phi^2\,.
  \]
  In conclusion,
  \[
    \Gamma_3=2\Big(\zeta\,\nabla\zeta\cdot \II_E^\Phi[\nabla\Phi]\,\Phi+\zeta (\nabla\zeta\cdot\nabla^2\Phi[a])\,\Phi^2
    -\zeta H_E^\Phi\,(\nabla\zeta\cdot\nabla\Phi)\,\Phi\Big)\,,
  \]
  so that
  \[
  \Gamma_1+\Gamma_2+\Gamma_3=\Big(\nabla\Phi\cdot\nabla g-\tr[(\II_E^\Phi)^2]\Big)\,\zeta^2\,\Phi\,.
  \]
  On noticing that $\Gamma_4=\Phi^2\,\nabla^2 \Phi\,\big[\nabla\zeta]\cdot \nabla\zeta$, we conclude the proof of \eqref{second variation}. By \eqref{elliptic}, one has $\nabla^2\Phi\ge(1/\lambda)\Id_{T_x(R_A(E))}$ for every $x\in R_A(E)$, and thus $\tr[(\II_E^\Phi)^2]\ge \lambda^{-2}\,|\II_E|^2$. Hence, \eqref{second variation} implies \eqref{second variation inequality appendix}.
  \end{proof}

%

\def\cprime{$'$}

\end{document}

%% file: min.pstex_t
\begin{picture}(0,0)%
\includegraphics{min.eps}%
\end{picture}%
\setlength{\unitlength}{3947sp}%
\begingroup\makeatletter\ifx\SetFigFont\undefined%
\gdef\SetFigFont#1#2#3#4#5{%
  \reset@font\fontsize{#1}{#2pt}%
  \fontfamily{#3}\fontseries{#4}\fontshape{#5}%
  \selectfont}%
\fi\endgroup%
\begin{picture}(2421,1899)(507,-1536)
\put(670,-559){\makebox(0,0)[lb]{\smash{{\SetFigFont{10}{12.0}{\rmdefault}{\mddefault}{\updefault}{\color[rgb]{0,0,0}$A$}%
}}}}
\put(2609,-1405){\makebox(0,0)[lb]{\smash{{\SetFigFont{10}{12.0}{\rmdefault}{\mddefault}{\updefault}{\color[rgb]{0,0,0}$H$}%
}}}}
\put(713,-1329){\makebox(0,0)[lb]{\smash{{\SetFigFont{8}{9.6}{\rmdefault}{\mddefault}{\updefault}{\color[rgb]{0,0,0}$\nu_H$}%
}}}}
\put(1651,-515){\makebox(0,0)[lb]{\smash{{\SetFigFont{10}{12.0}{\rmdefault}{\mddefault}{\updefault}{\color[rgb]{0,0,0}$E$}%
}}}}
\put(1572,-126){\makebox(0,0)[lb]{\smash{{\SetFigFont{8}{9.6}{\rmdefault}{\mddefault}{\updefault}{\color[rgb]{0,0,0}$B_{x,r}$}%
}}}}
\end{picture}%

%% file: crelle_maggi_revised.bbl
\def\cprime{$'$}
\begin{thebibliography}{DPM14}

\bibitem[DPM14]{dephilippismaggiCAPILLARI}
G.~De~Philippis and F.~Maggi.
\newblock Regularity of free boundaries in anisotropic capillarity problems and
  the validity of {Y}oung's law.
\newblock 2014.
\newblock preprint arXiv:1402.0549.

\bibitem[EG92]{EvansGariepyBOOK}
L.~C. Evans and R.~F. Gariepy.
\newblock {\em Measure theory and fine properties of functions}.
\newblock Studies in Advanced Mathematics. CRC Press, Boca Raton, FL, 1992.

\bibitem[Gr{\"u}87]{gruter2}
M.~Gr{\"u}ter.
\newblock Optimal regularity for codimension one minimal surfaces with a free
  boundary.
\newblock {\em Manuscripta Math.}, 58(3):295--343, 1987.

\bibitem[Mag12]{maggiBOOK}
F.~Maggi.
\newblock {\em Sets of finite perimeter and geometric variational problems},
  volume 135 of {\em Cambridge Studies in Advanced Mathematics}.
\newblock Cambridge University Press, Cambridge, 2012.
\newblock An introduction to {G}eometric {M}easure {T}heory.

\bibitem[Mor91]{MorganClifford}
F.~Morgan.
\newblock The cone over the {C}lifford torus in {${\bf R}^4$} is
  {$\Phi$}-minimizing.
\newblock {\em Math. Ann.}, 289(2):341--354, 1991.

\bibitem[Sim83]{SimonLN}
L.~Simon.
\newblock {\em Lectures on geometric measure theory}, volume~3 of {\em
  Proceedings of the Centre for Mathematical Analysis}.
\newblock Australian National University, Centre for Mathematical Analysis,
  Canberra, 1983.

\bibitem[SSA77]{schoensimonalmgren}
R.~Schoen, L.~Simon, and F.~J.~Jr. Almgren.
\newblock Regularity and singularity estimates on hypersurfaces minimizing
  parametric elliptic variational integrals. {I}, {II}.
\newblock {\em Acta Math.}, 139(3-4):217--265, 1977.

\bibitem[SZ99]{SternZum2}
P.~Sternberg and K.~Zumbrun.
\newblock On the connectedness of boundaries of sets minimizing perimeter
  subject to a volume constraint.
\newblock {\em Comm. Anal. Geom.}, 7(1):199--220, 1999.

\bibitem[Whi86]{whitemodp}
B.~White.
\newblock A regularity theorem for minimizing hypersurfaces modulo {$p$}.
\newblock In {\em Geometric measure theory and the calculus of variations
  ({A}rcata, {C}alif., 1984)}, volume~44 of {\em Proc. Sympos. Pure Math.},
  pages 413--427. Amer. Math. Soc., Providence, RI, 1986.

\end{thebibliography}
